\documentclass[a4paper,11pt]{article}

\usepackage{a4wide}
\usepackage[english]{babel}
\usepackage{picins}
\usepackage{amsfonts}
\usepackage{amsmath}
\usepackage{graphicx}
\usepackage[colorlinks,bookmarks=true]{hyperref}
\usepackage{amsthm}
\usepackage{enumitem}

\newtheorem{lemma}{Lemma}[section]

\newtheorem{theorem}{Theorem}[section]
\newtheorem{proposition}{Proposition}[section]

\newtheorem{corollary}{Corollary}[section]

\newcommand{\E}{\mathbb{E}}

\newcommand{\ch}{\text{ch}}
\newcommand{\sh}{\text{sh}}
\newcommand{\e}{\text{exp}}

\def\be{\begin{eqnarray}}
\def\ee{\end{eqnarray}}
\def\b*{\begin{eqnarray*}}
\def\eq*{\end{eqnarray*}}
\def\beq{\begin{equation}}
\def\eeq{\end{equation}}

\title{On the expectation of normalized Brownian functionals up to first hitting times}

\author{Romuald Elie$^{1}$, Mathieu Rosenbaum$^{2}$ and Marc Yor$^{2}$\\$~~$\\
$^{1}$ LAMA, University Paris-Est Marne-La-Vall\'ee\\
$^{2}$ LPMA, University Pierre et Marie Curie (Paris 6)}

\begin{document}

\maketitle

\begin{abstract}
\noindent Let $B$ be a Brownian motion and $T_1$ its first hitting time of the level $1$. For $U$ a uniform random variable independent of $B$, we study in depth the distribution of $B_{UT_1}/\sqrt{T_1}$, that is the rescaled Brownian motion sampled at uniform time. In particular, we show that this variable is centered.
\end{abstract}


\section{Introduction}\label{intro}

In this paper, we study the expectations of the random variables $A_a^{(m)}$ and $\tilde A_a^{(m)}$ defined for $a>0$ and $m\geq 0$ by
$$A_a^{(m)}=\frac{1}{T_a^{1+m/2}}\int_0^{T_a}|B_s|^m\text{sgn}(B_s) ds,~~\tilde A_a^{(m)}=\frac{1}{T_a^{1+m/2}}\int_0^{T_a}|B_s|^mds,$$
where $B$ is a Brownian motion and $T_a$ denotes the first hitting time of the level $a$ by $B$. 
First, remark that $T_a/a^2$ is the first hitting time of $a$ by $(B_{a^2s})$. Therefore, the scaling property of the Brownian motion implies that the laws of $A_a^{(m)}$ and $\tilde A_a^{(m)}$ do not depend on $a$.\\

\noindent To fix ideas, let us now focus in this introduction on the variables $A_a^{(m)}$. These variables are clearly asymmetric functionals of the Brownian motion. Nevertheless, we may wonder whether there exist values of $m$ such that $A_a^{(m)}$ is centered (we will show later that these variables have moments of all orders). Indeed, consider for example the case where $m$ is an odd integer: using a symmetry argument, it is clear that
$$\E[A_a^{(m)}]=-\E[A_{-a}^{(m)}],$$ where $A_{-a}^{(m)}$ is obviously defined. Since these two quantities do not depend on $a$, we get a given value, say $v_m$, for the expectations when the barrier is positive and $-v_m$ when it is negative. This somewhat suggests that $v_m$ may be equal to zero.\\

\noindent In fact, it turns out that the random variable $A_a^{(m)}$ is centered only for $m=1$.  This result has several interesting consequences. In particular, we show that it can be very simply interpreted in terms of the Brownian meander. Moreover, we prove that  the expectation of $A_a^{(m)}$ is negative for $m<1$ and positive for $m>1$.\\

\noindent  Finally, note that these expectations are closely connected with the random variable $\alpha$ defined by
$$\alpha=\frac{B_{UT_1}}{\sqrt{T_1}},$$
where $U$ is a uniform random variable, independent of $B$. For example, for $m$ an odd integer, the $m$-th moment of $\alpha$
is the expectation of $A_1^{(m)}$. This led us to give in Theorem \ref{density} the law of $\alpha$.\\

\noindent The paper is organized as follows. The specific case $m=1$ is treated in Section \ref{m1}. Our main theorem which provides the expectations of $A_{1}^{(m)}$ for any $m\geq 0$ is given in Section \ref{genm} together with its proof and some related results. The proofs of several technical results together with additional remarks are relegated to the four Appendices A, B, C and D.
 
\section{The case $m=1$}\label{m1} 

In this section, we state the nullity of the expectation of $A_1^{(1)}$, together with some associated results.
 
\subsection{Centering property in the case $m=1$} 

\begin{theorem}\label{theo1}
The random variable $A_1^{(1)}$ admits moments of all orders and is centered.
\end{theorem}
 
\noindent Theorem \ref{theo1} states that, as far as the expectation is concerned, between $0$ and $T_1$, the time spent by the Brownian motion in $(-\infty,0)$ is balanced by that spent in $[0,1]$. Again, it is tempting to deduce this result from the scaling and symmetry properties of $A_1^{(1)}$. However, Theorem \ref{theo2} will formalize that such intuition is wrong. Indeed, we will for example show that the expectation of $A_1^{(3)}$ is non zero, although it satisfies the same scaling and symmetry properties as $A_1^{(1)}$. In fact, we will see that the expectation of $A_1^{(m)}$ is strictly positive for $m>1$ and stricly negative for $m<1$.\\

\noindent Theorem \ref{theo1} can in fact be interpreted as a corollary of the general result given in Theorem \ref{theo2} below. However, using Williams time reversal theorem and some absolute continuity results for Bessel processes, a specific, elegant proof can be written for Theorem \ref{theo1}. So we give this proof in Appendix \ref{prooftheo1}. 

\subsection{More integrability properties for $A_1^{(1)}$ and connection with Knight's identity} 

Let $(L_t)_{t\geq 0}$ be the local time process at $0$ of the Brownian motion $B$ and set
$$\tau_l=\text{inf}\{t\geq 0,~L_t>l\},$$
for $l>0$. Recall that L\'evy's equivalence result, gives the following equality:
$$(|B_t|,L_t)_{t\geq 0}\underset{\mathcal{L}}{=}(S_t-B_t,S_t)_{t\geq 0},$$
with $S_t=\underset{s\leq t}{\text{sup }}B_s$ and $\underset{\mathcal{L}}{=}$ denotes equality in law, see \cite{revuz1999continuous}. Thus, we obtain that $A_1^{(1)}$ has the same law as the random variable $\zeta$ defined by  $$\zeta=\frac{1}{\tau_1^{3/2}}\int_0^{\tau_1}(L_u-|B_u|)du.$$
We obviously have
$$|\zeta|\leq \frac{1}{\sqrt{\tau_1}}+\frac{1}{\sqrt{\tau_1}}\underset{u\leq \tau_1}{\text{sup }}|B_u|.$$
On the one hand, it is well known that $1/\sqrt{\tau_1}$ follows the law of the absolute value of a standard Gaussian random variable. On the other hand, the celebrated Knight's identity states the following equality:
$$\frac{\tau_1}{(\underset{u\leq \tau_1}{\text{sup}}|B_u|)^2}\underset{\mathcal{L}}{=}T_2^3,$$
where $T_2^3=\text{inf}\{t,~R_t=2\},$ with $R$ a three dimensional Bessel process, see \cite{knight1988inverse}.
Using the scaling property of the three dimensional Bessel process, we easily get the equality:
$$T_2^3\underset{\mathcal{L}}{=}\frac{4}{(\underset{u\leq 1}{\text{sup }}R_u)^2}.$$ Therefore, we deduce that 
$$\frac{1}{\sqrt{\tau_1}}\underset{u\leq \tau_1}{\text{sup }}|B_u|\underset{\mathcal{L}}{=}\frac{1}{2}\underset{u\leq 1}{\text{sup }} R_u.$$
Hence, we easily deduce the following proposition:

\begin{proposition}
There exists $\varepsilon>0$ such that
$$\E[\emph{exp}\big(\varepsilon (A_1^{(1)})^2\big)]<+\infty.$$
\end{proposition}

\noindent We note that the same arguments yield that $A_1^{(m)}$ and $\tilde A_1^{(m)}$ admit moments of all orders. 

\subsection{Consequences of Theorem \ref{theo1} for the Bessel process, Brownian meander and Brownian bridge}

We give in this subsection some corollaries of Theorem \ref{theo1} involving very classical processes, namely the three dimensional
Bessel process, the Brownian meander, and the Brownian bridge. We start with a result about the three dimensional process, whose proof is given within the proof of Theorem \ref{theo1} in Appendix \ref{prooftheo1}. 

\begin{corollary}\label{bessel}
Let $(R_t)_{t\geq 0}$ denote a three dimensional Bessel process. We have
$$\E\big[\frac{1}{R_1^2}\int_0^1R_udu\big]=\sqrt{\frac{2}{\pi}}.$$
\end{corollary}

\noindent Thanks to Imhof's relation, see \cite{biane1987processus,imhof1984density}, we immediately deduce the following result from Corollary \ref{bessel}:

\begin{corollary}\label{mea}
Let $(m_t)_{t\leq 1}$ be the Brownian meander. We have
$$\E\big[\frac{1}{m_1}\int_0^1m_udu\big]=1.$$
\end{corollary}

\noindent We now give a corollary involving the Brownian bridge.

\begin{corollary}
Let $(b_t)_{t\leq 1}$ denote the Brownian bridge and
$(l_t)_{t\leq 1}$ its local time at zero. We have 
$$\E\big[\frac{1}{l_1}\int_0^1|b_u|du\big]=\E\big[\frac{1}{l_1}\int_0^1l_udu\big]=\frac{1}{2}.$$
\end{corollary}

\begin{proof}
From \cite{biane1988quelques}, we get the following equality:
$$(m_t,~t\leq 1)\underset{\mathcal{L}}{=}(|b_t|+l_t,~t\leq 1).$$
Thus, using Corollary \ref{mea}, we get
\begin{equation}\label{blone} \E\big[\frac{1}{l_1}\int_0^1(|b_u|+l_u)du\big]=1.\end{equation} 
Now remark that the process $(\hat{b}_t)=(b_{1-t})$ is also a Brownian bridge whose local time at time $t$, denoted by $\hat{l}_t$, satisfies
$$\hat{l}_t=l_1-l_{1-t}.$$ Consequently,
$$\E\big[\frac{1}{l_1}\int_0^1l_udu\big]=\E\big[\frac{1}{l_1}\int_0^1 \hat l_u du\big]=\E\big[\frac{1}{l_1}\int_0^1(l_1-l_u)du\big].$$ 
This implies
$$\E\big[\frac{1}{l_1}\int_0^1l_udu\big]=\frac{1}{2}$$ and therefore Equation \eqref{blone} provides
$$\E\big[\frac{1}{l_1}\int_0^1|b_u|du\big]=\frac{1}{2}.$$ 
\end{proof}

\noindent Finally, using a pathwise transformation between the meander and the Brownian excursion, see \cite{bertoin1994path}, Corollary \ref{mea} also enables to show the following result:

\begin{corollary}
Let $(e_t)_{t\leq 1}$ denote the standard Brownian excursion. We have 
$$\E\big[\int_0^1e_tdt\int_0^1\frac{1}{e_u}du\big]=\frac{3}{2}.$$
\end{corollary}

\subsection{The case of two barriers}

After the striking result given in Theorem \ref{theo1}, it is natural to wonder whether the expectation remains equal to zero
if $T_a$ is replaced by $T_{a,b}$, where $T_{a,b}$ is the first exit time of the interval $(-b,a)$, with $a>0$ and $b>0$. Indeed, remark that the random variable $A_{a,b}^{(1)}$ defined by 
$$A_{a,b}^{(1)}=\frac{1}{T_{a,b}^{3/2}}\int_0^{T_{a,b}}B_s ds,$$ 
still enjoys a scaling property in the sense that its law only depends on the ratio $b/a$. In fact, the following theorem states that the expectation is no longer zero in this case\footnote{However, note that from a numerical point of view, the obtained values for $\E[A_{a,b}^{(1)}]$ are systematically very small, for any $a$ and $b$.}:

\begin{theorem}\label{2bar}
Let $\lambda=b/a$. We have
$$\E[A_{a,b}^{(1)}]=\frac{1}{\sqrt{2\pi}}(1+\lambda)\int_0^\infty\frac{\delta}{\emph{sh}(\delta (1+\lambda))^2}(\lambda\emph{sh}(\delta)-\emph{sh}(\delta \lambda))d\delta.$$
In particular, $\E[A_{a,b}^{(1)}]\neq 0$ if $\lambda\notin \{0,1\}$. 
\end{theorem}

\noindent The proof of this result is given in Appendix \ref{proof2bar}. In fact a general formula for
$$\E\big[\frac{1}{T_{a,b}^{\theta}}\int_0^{T_{a,b}} B_s ds \big],$$ with $\theta>0$ is given within this proof.
Eventually, note that Theorem \ref{theo1} can also be recovered from Theorem \ref{2bar} letting the downward barrier tend to $-\infty$.
\section{The general case}\label{genm}
 
\subsection{Computation of the expectations} 
 
For $x\in\mathbb{R}$, we set $x^+=\text{max}(x,0)$ and $x^-=\text{max}(-x,0)$. For $m\geq 0$, we define 
$$A_{+}^{(m)}=\frac{1}{T_1^{1+m/2}}\int_0^{T_1}(B_s^+)^m ds,~~A_{-}^{(m)}=\frac{1}{T_1^{1+m/2}}\int_0^{T_1}(B_s^-)^m ds,$$
with the convention $0^0=0$. We also write
$$I_+^{(m)}=\E[A_{+}^{(m)}],~~I_-^{(m)}=\E[A_{-}^{(m)}].$$
and
$$I^{(m)}=I_+^{(m)}-I_-^{(m)}.$$
Furthermore, we note that $I_{\pm}^{(m)}$ is the moment of order $m$ of the random variable $\alpha^{\pm}$ where
$$\alpha=\frac{B_{UT_1}}{\sqrt{T_1}},$$
with $U$ a uniform random variable independent of the Brownian motion $B$. We study the variable $\alpha$ in more details in Section \ref{alpha}.\\

\noindent For $m\geq 0$, let
$$c_m=\frac{\Gamma(1+m)}{2^{m/2}\Gamma(1+m/2)}=\frac{1}{\sqrt{\pi}}2^{m/2}\Gamma(\frac{1+m}{2})=\E[|N|^m],$$
where $N$ is a standard Gaussian random variable and $\Gamma$ denotes the Gamma function. We have the following theorem:

\begin{theorem}\label{theo2}
Let $m\geq 0$ and introduce
$$\phi(m)=\int_0^2\frac{y^{m+1}}{1+y}dy.$$ The following formulas hold:
$$I_+^{(m)}=\frac{c_m}{2^{m+1}}\phi(m),~~I_-^{(m)}=\frac{c_m}{2^{m+1}}\emph{log}(3).$$
\end{theorem}
\noindent In particular, we note that $\phi(0)=2-\text{log}(3)$, $\phi(1)=\text{log}(3)$, $\phi(2)=8/3-\text{log}(3)$ and $\phi(3)=4/3+\text{log}(3)$. We give the proof of Theorem \ref{theo2} in Section \ref{prooftheo2}.
 
\subsection{Comments about Theorem \ref{theo2}}

\noindent $\bullet$ The function $\phi$ is well defined for $m\in(-2,+\infty)$ and satisfies $\phi(-1)=\phi(1)=\text{log}(3)$. Thus, we retrieve in Theorem \ref{theo2} the fact that $\E[A_1^{(1)}]=I_+^{(1)}-I_-^{(1)}=0.$\\ 

\noindent $\bullet$ We easily get that $\phi$ is twice differentiable and, for $m\ge 0$,
$$\phi'(m)=\int_0^2\frac{y^{m+1}\text{log}(y)}{1+y}dy,~~\phi''(m)=\int_0^2\frac{y^{m+1}(\text{log}(y))^2}{1+y}dy.$$
Hence $\phi$ is convex and furthermore, we show in Appendix \ref{fonction} that $\phi'(0)>0$. This implies that $\phi$ and $\phi'$ are increasing on $\mathbb{R}^+$. Hence, since
$$I^{(m)}=\frac{c_m}{2^{m+1}}\big(\phi(m)-\text{log}(3)\big),$$
we get $I^{(m)}>0$ for $m>1$ and $I^{(m)}<0$ for $m<1$. This can be interpreted as follows: from the point of view of $A_1^{(m)}$, for $m>1$, the time spent by the Brownian motion in $[0,1]$ is dominant whereas for $m<1$, the time spent in $(-\infty,0)$ is more important.\\

\noindent $\bullet$ Let $(L_t^x,~x\in\mathbb{R},~t\geq 0)$ denote the local time of the Brownian motion $B$. Within the proof of Theorem
\ref{theo2}, we are led to show the following interesting result:

\begin{proposition}\label{transfo}
Let $\mu>0$, $0<b<1$ and $x\geq 0$, we have
$$\E[L_{T_1}^b\emph{exp}(-\mu^2T_1/2)]=\frac{1}{\mu}\Big(\emph{exp}(-\mu)-\emph{exp}\big(-\mu(3-2b)\big)\Big)$$
and
$$\E[L_{T_1}^{-x}\emph{exp}(-\mu^2T_1/2)]=\frac{1}{\mu}\Big(\emph{exp}\big(-\mu(1+2x)\big)-\emph{exp}\big(-\mu(3+2x)\big)\Big).$$
\end{proposition}

\noindent We also give another proof of Proposition \ref{transfo}, based on the Ray-Knight theorem, in Appendix \ref{RK}.

\subsection{Uniform sampling up to hitting time}\label{alpha}

We now want to interpret Theorem \ref{theo2} as a result about sampling independently and uniformly  the
properly rescaled Brownian motion up to its first hitting time $T_1$. More precisely, let us introduce $(l_1^y,~y\in \mathbb{R}),$
the local time at time $1$ of the process
$$\big(\frac{B_{sT_1}}{\sqrt{T_1}},~s\leq 1\big).$$
Let $f$ be a Borel non negative function and $U$ a uniform random variable independent of any other random variable defined here. Using the occupation formula, we get
$$\E[f\big(\frac{B_{UT_1}}{\sqrt{T_1}}\big)]=\E[\int_0^1f\big(\frac{B_{sT_1}}{\sqrt{T_1}}\big)ds]=\int_{-\infty}^{+\infty}f(y)\E[l_1^y]dy.$$
Hence $h(y)=\E[l_1^y]$ is the density of $\alpha$ at point $y$. The following result is easily deduced from Theorem \ref{theo2}, by injectivity of the Mellin transform.

\begin{theorem}\label{density}
The density $h$ satisfies for $y\geq 0$
$$h(y)=\sqrt{\frac{2}{\pi}}\int_0^2\frac{1}{1+w}\emph{exp}(-2y^2/w^2)dw$$
and for $y\leq 0$
$$h(y)=\sqrt{\frac{2}{\pi}}\emph{log}(3)\emph{exp}(-2y^2).$$
\end{theorem}

\noindent Hence, conditional on $\alpha>0$, the law of $\alpha^+$ is a mixture of absolute Gaussian laws, whereas conditional on $\alpha<0$, $\alpha^-$ is distributed as the absolute value of a Gaussian random variable.\\

\noindent Remark that for $y\geq 0$, we have the obvious inequality
$$h(y)\leq \sqrt{\frac{2}{\pi}}\text{log}(3)\e(-y^2/2).$$ Therefore, we have the following corollary:

\begin{corollary}
For $\varepsilon<1/2$, the random variable $\alpha^+$ satisfies
$$\E[\emph{exp}\big(\varepsilon(\alpha^+)^2\big)]<+\infty.$$
\end{corollary}

\noindent In fact, thanks to Proposition \ref{transfo}, we can even provide the density at point $y$ of $\alpha$ conditional on $T_1=t$. We denote this density by $h(y,t)$. Obvious relations between $(l_1^y)$ and $(L_t^y)$ yield
$$h(y,t)=\E_{T_1=t}[l_1^y]=\frac{1}{\sqrt{t}}\E_{T_1=t}[L_t^{y\sqrt{t}}].$$ 
We easily obtain the following corollary from Proposition \ref{transfo}:

\begin{corollary}\label{cond}
The conditional density $h(y,t)$ satisfies for $0\leq y\sqrt{t}\leq 1$,
$$h(y,t)\emph{exp}\big(-1/(2t)\big)t^{-1/2}=\emph{exp}\big(-1/(2t)\big)-\emph{exp}\big(-(3-2y\sqrt{t})^2/(2t)\big)$$
and for $x\geq 0$
$$h(-x,t)\emph{exp}\big(-1/(2t)\big)t^{-1/2}=\emph{exp}\big(-(1+2x\sqrt{t})^2/(2t)\big)-\emph{exp}\big(-(1+3x\sqrt{t})^2/(2t)\big).$$
\end{corollary}

\subsection{Interpretation in terms of the Brownian meander}

In the same spirit as in Corollary \ref{mea}, we can give an interpretation of Theorem \ref{theo2} in terms of the Brownian meander.
Using Williams theorem in the same way as in the proof of Theorem \ref{theo1}, see Appendix \ref{prooftheo1}, together with Imhof's relation, see \cite{biane1987processus,imhof1984density}, as already done for Corollary \ref{mea}, we get that for any non negative measurable functions $f$ and $g$,
$$\E\Big[\int_0^1f\big(\frac{B_{sT_1}}{\sqrt{T_1}}\big)dsg\big(\frac{1}{\sqrt{T_1}}\big)\Big]=\sqrt{\frac{2}{\pi}}\E\big[\int_0^1f(m_1-m_u)du\frac{g(m_1)}{m_1}\big],$$
where $m$ denotes the Brownian meander.
Let $U$ be a uniform random variable, independent of all other quantities. The last relation is equivalent to
$$\E\Big[f\big(\frac{B_{UT_1}}{\sqrt{T_1}}\big)g\big(\frac{1}{\sqrt{T_1}}\big)\Big]=\sqrt{\frac{2}{\pi}}\E\big[f(m_1-m_U)\frac{g(m_1)}{m_1}\big].$$
Thus, from Corollary \ref{cond}, we are able to compute the density of $m_U$ conditional on the value $m_1$. More precisely, we have the following theorem:
\begin{theorem}
Let $f$ be a Borel non negative function. We have
$$\E_{m_1=y}[f(m_U)]=\int_0^yh\big(y-z,\frac{1}{y^2}\big)f(z)dz+\int_y^{+\infty}h\big(-(z-y),\frac{1}{y^2}\big)f(z)dz.$$
\end{theorem}

\subsection{Future developments}

In this work, we have studied some properties of random sampling through the random variable
$$\alpha=\frac{B_{UT_1}}{\sqrt{T_1}}.$$
Another interesting variable is the variable $\beta$ defined by
$$\beta=\frac{B_{U\tau_1}}{\sqrt{\tau_1}},$$
with $\tau_l=\text{inf}\{t\geq 0,~L_t>l\}.$ In fact the associated process 
$$\big(\frac{B_{s\tau_1}}{\sqrt{\tau_1}},~s\leq 1\big)$$ is called pseudo Brownian bridge and
has been considered more explicitly in the literature than 
$$\big(\frac{B_{sT_1}}{\sqrt{T_1}},~s\leq 1\big).$$
In particular, it enjoys some absolute continuity property with respect to the standard Brownian bridge, see \cite{biane1987processus}.
We intend to present results related to $\beta$ in a forthcoming work, in a way which will help us to recover the interesting law of $\alpha$. For now, we only mention that $\beta$ is distributed as $N/2$, where $N$ is a standard Gaussian random variable.

\subsection{Proof of Theorem \ref{theo2}}\label{prooftheo2}
Let $m\ge 0$. We split the proof into several steps.
\subsubsection*{Step 1: Introducing a natural measure}
 
First, let us remark that
\begin{align*}
I_{\pm}^{(m)}&=\frac{1}{\Gamma(1+m/2)}\E\big[\int_0^{+\infty}\lambda^{m/2}\text{exp}(-\lambda T_1)d\lambda\int_0^{T_1} (B_s^{\pm})^m ds\big]\\
&=\frac{1}{2^{m/2}\Gamma(1+m/2)}\E\big[\int_0^{+\infty}\mu^{1+m}\text{exp}(-\mu^2 T_1/2)d\mu\int_0^{T_1} (B_s^{\pm})^mds\big].
\end{align*}
Hence, it is natural to introduce for $\mu\ge 0$ the measure $I_{\mu}$, which to a positive function $\psi$ associates
$$I_{\mu}(\psi)=\E\big[\int_0^{T_1}\psi(B_s)\text{exp}(-\mu^2 T_1/2)ds\big]=e^{-\mu}\E\big[\int_0^{T_1}\psi(B_s)\text{exp}(\mu-\mu^2 T_1/2)ds\big].$$

\subsubsection*{Step 2: Computation of $I_{\mu}(\psi)$}

Let $(S_s)=(\underset{u\leq s}{\text{sup }}B_u)$. Using the martingale property of the process $\e(\mu B_s-\mu^2s/2)$, we get
$$I_{\mu}(\psi)=e^{-\mu}\E\big[\int_0^{+\infty}\psi(B_s)\mathrm{1}_{\{S_s<1\}}\text{exp}(\mu B_s-\mu^2 s/2)ds\big].$$
We now use the following well known formula, see for example \cite{pitman1999distribution}: for $s> 0$ and $b\in\mathbb{R}$,
$$\mathbb{P}[S_s<1|B_s=b]=1-\e\big(-\frac{2}{s}(1-b)^+\big).$$
It implies that $I_{\mu}(\psi)$ is equal to
$$e^{-\mu}\int_0^{+\infty}\text{exp}(-\mu^2 s/2)ds\int_{-\infty}^1e^{\mu b}\psi(b)\frac{1}{\sqrt{2\pi s}}\e\big(-b^2/(2s)\big)\Big(1-\e\big(-\frac{2}{s}(1-b)\big)\Big)db,$$
which can be rewritten
$$e^{-\mu}\int_{-\infty}^1e^{\mu b}\psi(b)db\int_0^{+\infty}\text{exp}(-\mu^2 s/2)\frac{1}{\sqrt{2\pi s}}\Big(\e\big(-b^2/(2s)\big)-\e\big(-(2-b)^2/(2s)\big)\Big)ds.$$
Then, using the density and the value of the first moment of an inverse Gaussian random variable, we get that for $\mu>0$ and $y\in\mathbb{R}$,
$$\int_0^{+\infty}\frac{1}{\sqrt{2\pi s}}\e\big(-y^2/(2s)-\mu^2s/2\big)ds=\frac{1}{\mu}\e(-\mu|y|).$$
From this, we deduce that when the support of $\psi$ is included in $[0,1]$,
\begin{equation}\label{eqfond1}
I_{\mu}(\psi)=\frac{1}{\mu}\int_0^1\psi(b)\Big(\e(-\mu)-\e\big(-\mu(3-2b)\big)\Big)db,
\end{equation}
and when the support of $\psi$ is included in $(-\infty,0)$,
\begin{equation}\label{eqfond2}
I_{\mu}(\psi)=\frac{1}{\mu}\int_0^{+\infty}\psi(-x)\Big(\e\big(-\mu(1+2x)\big)-\e\big(-\mu(3+2x)\big)\Big)dx.
\end{equation}
Remark here that Proposition \ref{transfo} immediately follows from Equation \eqref{eqfond1} and Equation \eqref{eqfond2}.
\subsubsection*{Step 3: End of the proof of Theorem \ref{theo2}}

We end the proof of Theorem \ref{theo2} in this final step. We start with the following elementary lemma:

\begin{lemma}\label{lab} For $a>0$, $b>0$ and $m\geq 0$, we define
$$L(a,b,m)=\int_0^{+\infty}y^m\big(\frac{1}{(a+y)^{m+1}}-\frac{1}{(b+y)^{m+1}}\big)dy.$$
The following equality holds: $$L(a,b,m)=\emph{log}(b/a).$$
\end{lemma}

\begin{proof}
We have
\begin{align*}
L(a,b,m)&=\underset{n\rightarrow +\infty}{\text{lim}}\int_0^{n}y^m\big(\frac{1}{(a+y)^{m+1}}-\frac{1}{(b+y)^{m+1}}\big)dy\\
&=\underset{n\rightarrow +\infty}{\text{lim}}\big(\int_0^{n/a}\frac{y^m}{(1+y)^{m+1}}dy-\int_0^{n/b}\frac{y^m}{(1+y)^{m+1}}dy\big)\\
&=\underset{n\rightarrow +\infty}{\text{lim}}\int_{1/b}^{1/a}\frac{n^{m+1}y^m}{(1+ny)^{m+1}}dy=\text{log}(b/a).
\end{align*}
\end{proof}

\noindent Now we take $\psi(x)=(x^{\pm})^m$ in Equation \eqref{eqfond1} and Equation \eqref{eqfond2}. Integrating in $\mu$, we easily derive
$$I_{+}^{(m)}=\frac{\Gamma(1+m)}{2^{m/2}\Gamma(1+m/2)}\int_0^1b^m\big(1-\frac{1}{(3-2b)^{m+1}}\big)db$$ 
and
$$I_{-}^{(m)}=\frac{\Gamma(1+m)}{2^{m/2}\Gamma(1+m/2)}\int_0^{+\infty}x^m\big(\frac{1}{(1+2x)^{m+1}}-\frac{1}{(3+2x)^{m+1}}\big)dx.$$
Applying Lemma \ref{lab}, we obtain the result for $I_{-}^{(m)}$. For $I_{+}^{(m)}$, we write
$$I_{+}^{(m)}=\frac{\Gamma(1+m)}{2^{m/2}\Gamma(1+m/2)}\big(\int_0^1b^mdb-\int_0^1\frac{1}{(3-2b)}\frac{b^m}{(3-2b)^{m}}db\big).$$ 
Then we use the change of variable $y=b/(3-2b)$ in the second integral in order to retrieve the expression of $I_{+}^{(m)}$ given in Theorem \ref{theo2}.
\newpage
\appendix

\section*{Appendices}

\section{Proof of Theorem \ref{theo1}}\label{prooftheo1}

Theorem \ref{theo1} can be seen as a particular case of Theorem \ref{theo2}. Nevertheless, we give here a specific proof for this theorem which is interesting on its own. We split it in several steps.

\subsection*{Step 1: Time reversal}

\noindent Let us recall Williams time reversal theorem, see for example \cite{revuz1999continuous}. We have the following equality:
$$(1-B_{T_1-u},~u\leq T_1)\underset{\mathcal{L}}{=}(R_u,~u\leq \gamma),$$
where $R$ denotes a three dimensional Bessel process starting from $0$ and $\gamma$ is its last passage time at level 1:
$$\gamma=\text{sup}\{t\geq 0,~R_t=1\}.$$
Consequently, since
$$A_1^{(1)}=\frac{1}{T_1^{3/2}}\int_0^{T_1}\big(1-(1-B_{T_1-s})\big) ds,$$
it has the same law as
\begin{equation}\label{dec}
\frac{1}{\gamma^{3/2}}\int_0^{\gamma}(1-R_u)du=\frac{1}{\sqrt{\gamma}}-\int_0^{1}\frac{R_{v\gamma}}{\sqrt{\gamma}}dv.
\end{equation}

\subsection*{Step 2: Moments}

\noindent We now show that $A_1^{(1)}$ has moments of any order. First recall the following equalities:
$$\frac{1}{\sqrt\gamma}\underset{\mathcal{L}}{=}\frac{1}{\sqrt{T_1}}\underset{\mathcal{L}}{=}|B_1|.$$
Thus, $\frac{1}{\sqrt\gamma}$ has moments of any order and therefore it is enough to prove the integrability of $\xi^r$, for any $r>0$, with
$$\xi=\int_0^{1}\frac{R_{v\gamma}}{\sqrt{\gamma}}dv.$$
Such integrability result will be deduced from the following absolute continuity relation that can be found in \cite{biane1987processus}:
\begin{lemma}\label{bly}
For any Borel functional $F$ from $(\mathbb{C}[0,1],\mathbb{R}^+)$ into $\mathbb{R}^+$,
$$\E\big[F\big(\frac{R_{u\gamma}}{\sqrt{\gamma}},~u\leq 1\big)\big]=\E\big[F\big(R_{u},~u\leq 1\big)\frac{1}{R_1^2}\big].$$
\end{lemma}
\noindent Now take $r>0$, $1<p<3/2$ and $q$ such that $1/p+1/q=1$. From Lemma \ref{bly} together with H\"older inequality, we obtain
$$\E[(\xi^r)]=\E\big[\big(\int_0^1 R_u du\big)^r\frac{1}{R_1^2}\big]\leq\big(\E\big[\big(\int_0^1 R_u du\big)^{rq}\big]\big)^{1/q}\big(\E\big[\frac{1}{R_1^{2p}}\big]\big)^{1/p}.$$
The first expectation on the right hand side of the last inequality is obviously finite. For the second one, recall that $R_1^2$ has the distribution of $2Z$, with $Z$ following a gamma law with parameter $3/2$. Therefore, the second expectation is also finite since $p<3/2$.

\subsection*{Step 3: Centering property}

We end the proof of Theorem \ref{theo1} in this step. We start with the following technical lemma. 

\begin{lemma}\label{tech} Let $a>0$. We have
$$\E[R_1\emph{exp}(-R_1^2a/2)]=\frac{\sqrt{2}}{\Gamma(3/2)(1+a)^2}.$$
\end{lemma}
\begin{proof}
Using again that $R_1^2$ has the distribution of $2Z$, with $Z$ following a gamma law with parameter $3/2$, we can write
$$\E[R_1\text{exp}(-R_1^2a/2)]=\frac{\sqrt{2}}{\Gamma(3/2)}\int_0^{+\infty}x\text{exp}(-x(1+a))dx.$$
The result follows easily from this equality.
\end{proof}

\noindent We now prove that $\E[A_1^{(1)}]=0$. From Equation \eqref{dec} and Lemma \ref{bly}, using the fact that $\E[1/\sqrt{\gamma}]=\E[|B_1|]=\sqrt{2/\pi}$, this is equivalent to prove the following lemma:

\begin{lemma}\label{specialrel} We have
$$\E\big[\big(\int_0^1 R_u du\big)\frac{1}{R_1^2}\big]=\sqrt{\frac{2}{\pi}}.$$
\end{lemma}

\begin{proof}
First, using Markov property, we get
$$\E\big[\frac{R_u}{R_1^2}\big]=\E\big[R_u\E_{R_u}[\frac{1}{R_{1-u}^2}]\big],$$
where $\E_{r}$ denotes the expectation of a three dimensional Bessel process starting from point $r$. From Proposition 2, page 99, in \cite{yor2001exponential}, we know that
$$\E_r\big[\frac{1}{R_t^2}\big]=\int_0^{1/(2t)}\text{exp}(-r^2v)(1-2tv)^{-1/2}dv.$$
Thus, using the last equality together with a change of variable and the scaling property of the Bessel process, we get
$$\E\big[\frac{R_u}{R_1^2}\big]=\sqrt{u}\E\big[R_1\int_0^1\frac{(1-w)^{-1/2}}{2(1-u)}\text{exp}\big(-\frac{R_1^2uw}{2(1-u)}\big)dw\big].$$
From Lemma \ref{tech}, we obtain
$$
\E\big[\frac{R_u}{R_1^2}\big]=\frac{\sqrt{u}(1-u)}{\sqrt{2}\Gamma(3/2)}\int_0^1\frac{1}{\sqrt{x}(1-ux)^2}dx=\frac{(1-u)}{\sqrt{2}\Gamma(3/2)}\int_0^u\frac{1}{\sqrt{y}(1-y)^2}dy.
$$
Using Fubini's theorem when integrating in $u$ from $0$ to $1$, and remarking that $\Gamma(3/2)=\sqrt{\pi}/2$, we easily conclude the proof of Lemma \ref{specialrel} and so the proof of Theorem \ref{theo1}. 
\end{proof}

\section{Proof of Theorem \ref{2bar}}\label{proof2bar}

Let $a>0$, $b>0$ and $\theta>0$. In this section, we consider $\psi(a,b,\theta)$ defined by
\begin{equation*} 
\psi(a,b,\theta)=\E\left[\frac{1}{\tau^{\theta}}\int_0^\tau B_s ds \right], 
\end{equation*}
where $\tau$ is the exit time of the interval $(-b,a)$ by the Brownian motion $B$.

\subsection{General result}
We start with a general result. We give here a representation of $\psi(a,b,\theta)$ in term of a Lebesgue integral.
Let $\delta>0$, $a>0$, $b>0$ and $p>-1$. Recall that $c_p$ denotes the $p-$th absolute moment of a standard Gaussian random variable and define $\phi_{\delta}(a,b,p)$
by
$$\phi_{\delta}(a,b,p)=ab+b^2\big(p-1-(p-2)\ch(\delta(a+b))\big).$$
We have the following result.
\begin{theorem}\label{theoapp}
Let $\theta>0$. We have
$$
\psi(a,b,\theta)=\frac{\sqrt{2}}{\sqrt{\pi}m_{2\theta-1}}\int_0^\infty \delta^{2\theta-1} E_\delta d\delta,
$$
with
$$
 E_\delta
 = \frac{b\emph{sh}(\delta a)-a\emph{sh}(\delta b)}{2\delta^2\emph{sh}(\delta(a+b))} + \frac{(a^2 \emph{sh}(\delta b)-  b^2 \emph{sh}(\delta a))(\emph{ch}(\delta (a+b)) -1) }{2\delta\emph{sh}(\delta(a+b))^2}.
$$
For $\theta\neq 1$, another representation for $\psi(a,b,\theta)$ is
$$\frac{\sqrt{2}}{\sqrt{\pi}m_{2\theta-1}}\int_0^\infty\frac{\delta^{2\theta-2}}{4(\theta-1)\emph{sh}(\delta(a+b))^2}\big(\emph{sh}(\delta a)\phi_{\delta}(a,b,2\theta-1)-\emph{sh}(\delta b)\phi_{\delta}(b,a,2\theta-1)\big)d\delta.$$
\end{theorem}

\begin{proof}
Our proof is based on Feynman-Kac formula, see for example \cite{borodin2002handbook}. Note that in \cite{lachal1991premier}, the author used this formula in order to derive the joint Laplace transform of $(\tau,\int_0^{\tau}B_s ds)$. We propose here a specific method for our problem. We introduce the function
\b*
g:(x,\delta,\rho) \mapsto \E_x\left[e^{ -(\delta^2/2)\tau +\rho \int_0^\tau B_s ds} \right]\;.
\eq*
 By Feynman-Kac formula, $g$ solves on $(-b,a)$
\b*
g_{xx}(x,\delta,\rho) -(\delta^2-2\rho x) g(x,\delta,\rho)=0\;, \qquad \mbox{with } g(a,.)=g(-b,.)=1.
\eq*
For $\rho=0$, we denote $g^0:(x,\delta)\mapsto g(x,\delta,0)$ which solves on $(-b,a)$
 \b*
g^0_{xx}(x,\delta,\rho) -\delta^2 g^0(x,\delta,\rho)=0, \qquad \mbox{with } 
g^0(a,.)=g^0(-b,.)=1.
\eq* 
Thus, $g^0$ is of the form
$$
g^0(x,\delta)=A_\delta \ch(\delta x) + B_\delta \sh(\delta x).
$$
Differentiating the dynamics of $g$ with respect to $\rho$ and introducing 
$$f:(x,\delta) \mapsto g_\rho(x,\delta,0),$$ we observe that $f$ solves on $(-b,a)$
\b*
f_{xx}(x,\delta) -\delta^2 f(x,\delta) + 2x g^0(x,\delta)=0,
\qquad \mbox{with } f(a,.)=f(-b,.)=0.
\eq*
 Furthermore, by definition of $g$, $f$ satisfies
$$
f(x,\delta)=\E_x\left[e^{ -(\delta^2/2)\tau} \int_0^\tau B_s ds \right].
$$
 Due to its dynamics, we get that $f$ is of the form
$$
f(x,\delta)=E_\delta \ch(\delta x) + F_\delta \sh(\delta x) +f^0(x,\delta),
$$
where $f^0$ is a particular solution of the ODE of interest. Applying the variation of the constant method, we look for $f^0$ of the form 
$$
f^0(x,\delta)=C_\delta(x) \ch(\delta x) + D_\delta(x) \sh(\delta x), 
$$
so that $f$ rewrites
\b*
f(x,\delta)=(E_\delta+C_\delta (x)) \ch(\delta x) + (F_\delta+D_\delta(x)) \sh(\delta x). 
\eq*
This function $f$ is of particular interest since for $p>-1$,
$$
\E\left[ \tau^{-(p+1)/2} \int_0^\tau B_s ds \right] 
$$
is equal to
\b*
\frac{\sqrt{2}}{\sqrt{\pi}c_p} \E\left[ \left(\int_0^\infty \delta^p e^{-(\delta^2/2)\tau} d\delta\right)\left(\int_0^\tau B_s ds\right) \right]= \frac{\sqrt{2}}{\sqrt{\pi}c_p}\int_0^\infty f(0,\delta)\delta^p d\delta.
\eq*

\noindent Hence, denoting $p=2\theta-1$, the first part of Theorem \ref{theoapp} boils down to the computation of 
 \begin{equation}\label{quant}
 \int_0^\infty  f(0,\delta) \delta^p d\delta \;=\; \int_0^\infty (E_\delta+C_\delta (0))\delta^p d\delta,
 \end{equation}
so that we need to identify $A_\delta,B_\delta,C_\delta(.),D_\delta(.),E_\delta$ and $F_\delta$.\\
 
\noindent Observe that from the boundary conditions $g^0(a,.)=g^0(-b,.)=1$, we obtain that $A_\delta$ and $B_\delta$ satisfy
 \b*
 A_\delta \ch(\delta a) + B_\delta \sh(\delta a) \;=\; 1,\quad
 A_\delta \ch(\delta b) - B_\delta \sh(\delta b) \;=\; 1.
 \eq*
 
\noindent We recall now for later use the classical $\ch$ and $\sh$ formulas:
 \b*
 \ch(x)\ch(y) {\pm} \sh(x)\sh(y) &=& \ch(x{\pm}y),\\
 \ch(x)\sh(y) {\pm} \sh(x)\ch(y) &=& \sh(x{\pm}y).\\
 \eq*
We deduce that $A_\delta$ and $B_\delta$ are given by :
\begin{equation}\label{ABdelta}
A_\delta \;=\; \frac{\sh(\delta b)  + \sh(\delta a)}{\sh(\delta(a+b))}, \qquad
B_\delta \;=\; \frac{\ch(\delta b)  - \ch(\delta a)}{\sh(\delta(a+b))}.
\end{equation}

\noindent Similarly we can compute $E_\delta$ and $F_\delta$ in terms of $C_\delta(.)$ and $D_\delta(.)$. Indeed, the boundary conditions of $f$ imply
 \b*
 E_\delta \ch(\delta a) + F_\delta \sh(\delta a) &=& -  C_\delta (a) \ch(\delta a) -  D_\delta(a) \sh(\delta a),\\
 E_\delta \ch(\delta b) - F_\delta \sh(\delta b) &=& -  C_\delta (-b) \ch(\delta b) +  D_\delta(-b) \sh(\delta b).
 \eq*

\noindent Consequently, we get that $E_\delta\sh(\delta(a+b))$ is equal to
\b*
- C_\delta (a) \ch(\delta a)\sh(\delta b)&-&  D_\delta(a) \sh(\delta a)\sh(\delta b)\\
-  C_\delta (-b) \ch(\delta b)\sh(\delta a)&+&  D_\delta(-b) \sh(\delta b)\sh(\delta a)
\eq*
and $F_\delta\sh(\delta(a+b))$ to
\b*
- C_\delta (a) \ch(\delta a)\ch(\delta b)&-&  D_\delta(a) \sh(\delta a)\ch(\delta b)\\  
+  C_\delta (-b) \ch(\delta b)\ch(\delta a)&-&  D_\delta(-b) \sh(\delta b)\ch(\delta a).
\eq*
 It now remains to compute $C_\delta(.)$ and $D_\delta(.)$, which are both defined up to a constant and satisfy
\b*
f^0(x,\delta) &=& C_\delta(x) \ch(\delta x) + D_\delta(x) \sh(\delta x) \\
f^0_x(x,\delta) &=&  C_\delta(x) \delta \sh(\delta x) + D_\delta(x) \delta \ch(\delta x).
\eq*
Thus, since $f^0_{xx}-\delta^2 f^0 +2xg^0(x)=0$, $C_\delta'$ and $D_\delta'$ satisfy
\b*
 C_\delta'(x) \ch(\delta x) + D_\delta'(x) \sh(\delta x) &=& 0 \\
 C_\delta'(x) \delta \sh(\delta x) + D_\delta'(x) \delta \ch(\delta x) &=& -2x g^0(x). 
\eq*
Therefore, we get 
\b*
 C_\delta'(x) = \frac{2x g^0(x)}{\delta}\sh(\delta x),\qquad
 D_\delta'(x) = - \frac{2x g^0(x)}{\delta}\ch(\delta x).
\eq*
We now compute $C_\delta$, which is given by
\b*
 C_\delta(x) 
 &=& \frac{2}{\delta} \int_0^x t \left( A_\delta \ch(\delta t) +B_\delta \sh(\delta t)\right) \sh(\delta t) dt \\
 &=& \frac{A_\delta}{\delta} \int_0^x  t \sh(2\delta t) dt  + \frac{B_\delta}{\delta} \int_0^x t (\ch(2\delta t)-1) dt \\
 &=& \frac{A_\delta}{2\delta^2} \left( {x\ch(2\delta x)} - \frac{\sh(2\delta x)}{2\delta}  \right)  - \frac{B_\delta x^2}{2\delta} + \frac{B_\delta}{2\delta^2}  \left( {x\sh(2\delta x) } - \frac{\ch(2\delta x)}{2\delta} +\frac{1}{2\delta}\right).
\eq*

\noindent In the same way, $D_\delta$ is given by
\b*
 D_\delta(x) 
 &=& - \frac{2}{\delta} \int_0^x t \left( A_\delta \ch(\delta t) +B_\delta \sh(\delta t)\right) \ch(\delta t) dt \\
 &=& -\frac{B_\delta}{\delta} \int_0^x  t \sh(2\delta t) dt  - \frac{A_\delta}{\delta} \int_0^x t (1+\ch(2\delta t)) dt \\
 &=& -  \frac{B_\delta}{2\delta^2} \left( {x\ch(2\delta x)} - \frac{\sh(2\delta x)}{2\delta}  \right)  - \frac{A_\delta x^2}{2\delta} - \frac{A_\delta}{2\delta^2}  \left( {x\sh(2\delta x) } - \frac{\ch(2\delta x)}{2\delta} +\frac{1}{2\delta}\right).
\eq*

\noindent Since $C_\delta(0)=0$, observe that the quantity of interest \eqref{quant} rewrites 
$$\int_0^\infty \delta^p E_\delta d\delta,$$
where $E_\delta$ is given above as a function of $C_\delta(a)$, $C_\delta(-b)$, $D_\delta(a)$ and $D_\delta(-b)$. We now give an expression for
$$\sh(\delta(a+b)) E_\delta.$$
First recall that it is equal to
\begin{multline*}
- C_\delta (a) \ch(\delta a)\sh(\delta b) -  D_\delta(a) \sh(\delta a)\sh(\delta b)\\
-  C_\delta (-b) \ch(\delta b)\sh(\delta a) +  D_\delta(-b) \sh(\delta b)\sh(\delta a).
\end{multline*}

\noindent Plugging the values for the coefficients, this can be rewritten
\begin{align*}
&- \left[\frac{A_\delta}{2\delta^2} \left( {a\ch(2\delta a)} - \frac{\sh(2\delta a)}{2\delta}  \right)  - \frac{B_\delta a^2}{2\delta} + \frac{B_\delta}{2\delta^2}  \left( {a\sh(2\delta a) } - \frac{\ch(2\delta a)}{2\delta} +\frac{1}{2\delta}\right)\right] \ch (\delta a) \sh (\delta b)\\
&-  \left[-  \frac{B_\delta}{2\delta^2} \left( {a\ch(2\delta a)} - \frac{\sh(2\delta a)}{2\delta}  \right)  - \frac{A_\delta a^2}{2\delta} - \frac{A_\delta}{2\delta^2}  \left( {a\sh(2\delta a) } - \frac{\ch(2\delta a)}{2\delta} +\frac{1}{2\delta}\right)\right] \sh(\delta a)\sh(\delta b)\\
&-  \left[\frac{A_\delta}{2\delta^2} \left( {-b\ch(2\delta b)} + \frac{\sh(2\delta b)}{2\delta}  \right)  - \frac{B_\delta b^2}{2\delta} + \frac{B_\delta}{2\delta^2}  \left( {b\sh(2\delta b) } - \frac{\ch(2\delta b)}{2\delta} +\frac{1}{2\delta}\right)\right] \ch(\delta b)\sh(\delta a)\\
&+  \left[ -  \frac{B_\delta}{2\delta^2} \left( -{b\ch(2\delta b)} + \frac{\sh(2\delta b)}{2\delta}  \right)  - \frac{A_\delta b^2}{2\delta} - \frac{A_\delta}{2\delta^2}  \left( {b\sh(2\delta b) } - \frac{\ch(2\delta b)}{2\delta} +\frac{1}{2\delta}\right) \right] \sh(\delta b)\sh(\delta a),
\end{align*}
which leads to the expression: 
\begin{align*}
~&  
  \frac{1}{2\delta}\left[ a^2\sh(\delta b) \left( A_\delta \sh(\delta a) + B_\delta \ch (\delta a)  \right)  - b^2\sh(\delta a) \left( A_\delta\sh(\delta b)-B_\delta\ch(\delta b) \right)  \right] \\
&
  +\frac{a\sh(\delta b)}{2\delta^2} \left[A_\delta \left( \sh(2\delta a)\sh(\delta a) - \ch(2\delta a)\ch (\delta a) \right) +B_\delta\left(\ch(2\delta a) \sh(\delta a)-\sh(2\delta a)\ch (\delta a) \right)  \right]\\
&
  +\frac{b\sh(\delta a)}{2\delta^2} \left[-A_\delta \left( \sh(2\delta b)\sh(\delta b) - \ch(2\delta b)\ch (\delta b) \right) +B_\delta\left(\ch(2\delta b) \sh(\delta b)-\sh(2\delta b)\ch (\delta b) \right)  \right]\\
&
  +\frac{A_\delta}{4\delta^3} \left[ \sh(\delta b) \left(  \sh(2\delta a)\ch(\delta a) - \ch(2\delta a)\sh(\delta a) \right) -  \sh(\delta a) \left(  \sh(2\delta b)\ch(\delta b) - \ch(2\delta b)\sh(\delta b) \right) \right]\\
&  
  +\frac{B_\delta}{4\delta^3} \left[ \sh(\delta b) \left(  \ch(2\delta a)\ch(\delta a) - \sh(2\delta a)\sh(\delta a) \right) +  \sh(\delta a) \left(  \ch(2\delta b)\ch(\delta b) - \sh(2\delta b)\sh(\delta b) \right) \right]\\
& - \frac{B_\delta}{4\delta^3}\left[ \ch (\delta a) \sh (\delta b)+\ch(\delta b)\sh(\delta a) \right].
\end{align*} 
\noindent After obvious computations, we obtain that it is also equal to
\begin{align*}
~&\frac{1}{2\delta}\left[ a^2\sh(\delta b) \left( A_\delta \sh(\delta a) + B_\delta \ch (\delta a)  \right)  - b^2\sh(\delta a) \left( A_\delta\sh(\delta b)-B_\delta\ch(\delta b) \right)  \right] \\
&
  +\frac{a\sh(\delta b)}{2\delta^2} \left[-A_\delta \ch(\delta a)  -B_\delta\sh(\delta a)  \right] + \frac{b\sh(\delta a)}{2\delta^2} \left[A_\delta \ch(\delta b) -B_\delta \sh(\delta b)\right].  
\end{align*} 

\noindent By definition, $A_\delta \ch(\delta a) +B_\delta\sh(\delta a)=A_\delta \ch(\delta b) -B_\delta\sh(\delta b)=1$. Therefore, we get
 \begin{align*}
   \sh(\delta(a+b)) E_\delta&=
  \frac{1}{2\delta}\left[ a^2\sh(\delta b) \left( A_\delta \sh(\delta a) + B_\delta \ch (\delta a)  \right)  - b^2\sh(\delta a) \left( A_\delta\sh(\delta b)-B_\delta\ch(\delta b) \right)\right]\\
  ~&-\frac{a\sh(\delta b)-b\sh(\delta a)}{2\delta^2} .
 \end{align*}
Recall also that $A_\delta$ and $B_\delta$ are explicitly given by (\ref{ABdelta}) so that 
 \b*
 A_\delta \sh(\delta a) + B_\delta \ch (\delta a) 
 &=&
 \frac{\sh(\delta b)\sh(\delta a)  + \sh(\delta a)^2 + \ch(\delta b)\ch (\delta a)  - \ch(\delta a)^2}{\sh(\delta(a+b))}\\
 &=&
 \frac{\ch(\delta (a+b)) - 1}{\sh(\delta(a+b))}
 \eq*
 and
 \b*
 A_\delta \sh(\delta b) - B_\delta \ch (\delta b) 
 &=&
 \frac{\sh(\delta b)\sh(\delta a)  + \sh(\delta b)^2 + \ch(\delta b)\ch (\delta a)  - \ch(\delta b)^2}{\sh(\delta(a+b))}\\
 &=&
 \frac{\ch(\delta (a+b)) -1}{\sh(\delta(a+b))}. 
 \eq*
Plugging these expressions in the previous one provides
$$
 E_\delta
 = \frac{b\sh(\delta a)-a\sh(\delta b)}{2\delta^2\sh(\delta(a+b))} + \frac{(a^2 \sh(\delta b)-  b^2 \sh(\delta a))(\ch(\delta (a+b)) -1) }{2\delta\sh(\delta(a+b))^2}.
$$
Recalling that $p=2\theta-1$, this ends the proof of the first part of Theorem \ref{theoapp}.\\ 

\noindent We now give the proof of the second part. By integration by parts we get that
$$\int_0^\infty \delta^p E_\delta d\delta$$
is equal to 
\begin{align*}
 ~&\int_0^\infty\frac{b\sh(\delta a)-a\sh(\delta b)}{2\sh(\delta(a+b))}\delta^{p-2}  d\delta+ 
 \int_0^\infty \frac{(a^2 \sh(\delta b)-  b^2 \sh(\delta a))(\ch(\delta (a+b)) -1) }{2\sh(\delta(a+b))^2}\delta^{p-1} d\delta \\
 =&\left[ \frac{b\sh(\delta a)-a\sh(\delta b)}{2\sh(\delta(a+b))}\frac{\delta^{p-1}}{p-1} \right]_0^\infty  
 -  \int_0^\infty\frac{ba\ch(\delta a)-ba\ch(\delta b)}{2\sh(\delta(a+b))}\frac{\delta^{p-1}}{p-1}  d\delta  \\
 +&  \int_0^\infty    \frac{(b(a+b)\sh(\delta a)-a(a+b)\sh(\delta b))(\ch (\delta(a+b)))}{2\sh(\delta(a+b))^2}\frac{\delta^{p-1}}{p-1} d\delta\\
 +& 
 \int_0^\infty \frac{(a^2 \sh(\delta b)-  b^2 \sh(\delta a))(\ch(\delta (a+b)) -1) }{2\sh(\delta(a+b))^2}\delta^{p-1} d\delta.
\end{align*}
Then, we easily obtain that the last expression is equal to 
\begin{align*}
&-  \int_0^\infty\frac{ba\ch(\delta a)-ba\ch(\delta b)}{2\sh(\delta(a+b))}\frac{\delta^{p-1}}{p-1} d\delta  
+  \int_0^\infty    \frac{(ba\sh(\delta a)-ab\sh(\delta b))(\ch (\delta(a+b)))}{2\sh(\delta(a+b))^2}\frac{ \delta^{p-1}}{p-1} d\delta \\
 &+ \int_0^\infty \delta^{p-1}\frac{(a^2 \sh(\delta b)-  b^2 \sh(\delta a)) }{2\sh(\delta(a+b))^2}\ch (\delta(a+b))\frac{p-2}{p-1}d\delta- \int_0^\infty \frac{(a^2 \sh(\delta b)-  b^2 \sh(\delta a)) }{2\sh(\delta(a+b))^2}\delta^{p-1}d\delta.
\end{align*}
After obvious simplifications, this can be rewritten
\begin{align*}
& ba  \int_0^\infty\frac{-\sh(\delta b) +\sh(\delta a) }{2\sh(\delta(a+b))^2} \frac{\delta^{p-1}}{p-1} d\delta- \int_0^\infty \frac{(a^2 \sh(\delta b)-  b^2 \sh(\delta a)) }{2\sh(\delta(a+b))^2}\delta^{p-1}d\delta  \\
 &+ \int_0^\infty \delta^{p-1}\frac{(a^2 \sh(\delta b)-  b^2 \sh(\delta a)) }{2\sh(\delta(a+b))^2}\ch (\delta(a+b))\frac{p-2}{p-1}d\delta.
\end{align*}
Thus, using the function $\phi_{\delta}$ defined before Theorem \ref{theoapp}, we obtain
$$\int_0^\infty \delta^p E_\delta d\delta=\int_0^\infty\frac{\delta^{p-1}}{(p-1)2\sh(\delta(a+b))^2}(\sh(\delta a)\phi_{\delta}(a,b,p)-\sh(\delta b)\phi_{\delta}(b,a,p))d\delta.$$

\end{proof}

\subsection{Proof of Theorem \ref{2bar}}
We now give the proof of Theorem \ref{2bar}.
From Theorem \ref{theoapp}, we get
$$\psi(a,b,3/2)=\frac{1}{\sqrt{2\pi}}\int_0^\infty\frac{\delta}{\sh(\delta(a+b))^2}(\sh(\delta a)\phi_{\delta}(a,b,2)-\sh(\delta b)\phi_{\delta}(b,a,2)d\delta.$$
Then we use that
$$\phi_{\delta}(a,b,2)=ab+b^2,\quad \phi_{\delta}(b,a,2)=ab+a^2$$
in order to obtain
$$\psi(a,b,3/2)=\frac{1}{\sqrt{2\pi}}(a+b)\int_0^\infty\frac{\delta}{\sh(\delta(a+b))^2}(b\sh(\delta a)-a\sh(\delta b))d\delta.$$ Taking $\lambda = b/a$, we get
$$\psi(a,b,3/2)=\frac{1}{\sqrt{2\pi}}(1+\lambda)\int_0^\infty\frac{\delta a}{\sh(\delta a(1+\lambda))^2}(a\lambda\sh(\delta a)-a\sh(\delta a\lambda))d\delta.$$
We finally obtain the result after the change of variable $x=\delta a$.

\section{Some computations about the function $\phi$ defined in Theorem \ref{theo2}}\label{fonction}

Recall that the function $\phi$ is defined for $m>-2$ by 
$$\phi(m)=\int_0^2\frac{y^{m+1}}{1+y}dy.$$ 
We wish to compute
$$\phi'(0)=\int_0^2\frac{y\text{log}(y)}{1+y}dy.$$ 
We denote by $\text{Li}_2$ the dilogarithm function defined for $x$ such that $|x|\leq 1$ by
$$\text{Li}_2(x)=\sum_{n=1}^{+\infty}\frac{x^n}{n^2},$$
see \cite{andrews1999special} for more details. We start with the following general lemma:
\begin{lemma}\label{Li}
For $C\geq 1$, we define the function $\Delta$ by
$$\Delta(C)=\int_0^C\frac{y\emph{log}(y)}{1+y}dy.$$ We have $$\Delta(C)=C\emph{log}(C)-C-\big(\emph{log}(C)\big)\emph{log}(C+1)+\frac{\pi^2}{6}+\frac{1}{2}\big(\emph{log}(C)\big)^2+\emph{Li}_2\big(-\frac{1}{C}\big).$$
\end{lemma}
\begin{proof}
We get the equality of the two functions in Lemma \ref{Li} by showing that they have the same derivatives and that they coincide for $C=1$. To show the equality of the derivatives, after straightforward computations, we see that we need to prove that
$$-\text{log}\big(\frac{1}{C}+1\big)+\frac{1}{C}(\text{Li}_2)'\big(-\frac{1}{C}\big)$$
is equal to zero. Now we use the fact that for $|x|\leq 1$,
$$(\text{Li}_2)'(x)=-\frac{\text{log}(1-x)}{x},$$
see \cite{andrews1999special}, in order to get the result.\\

\noindent We now show that the values of the two functions in Lemma \ref{Li} coincide for $C=1$. We have
$$\int_0^1\frac{y\text{log}(y)}{1+y}dy=\sum_{n=0}^{+\infty}(-1)^n\int_0^1y^{1+n}\text{log}(y)dy.$$
Using integration by parts arguments, we deduce
$$\int_0^1\frac{y\text{log}(y)}{1+y}dy=-\sum_{n=2}^{+\infty}\frac{(-1)^n}{n^2}=-(\text{Li}_2(-1)+1).$$
We conclude using the fact that $\text{Li}_2(-1)=-\pi^2/12,$ see again \cite{andrews1999special}.  
\end{proof}

\noindent Recall that $\phi'(0)=\Delta(2)$. Using Lemma \ref{Li} together with the facts that $\text{Li}_2(-1/2)>-1/2$ and 
$$2\text{log}(2)-2-\text{log}(2)\text{log}(3)+\frac{\pi^2}{6}+\frac{1}{2}\big(\text{log}(2)\big)^2>\frac{1}{2},$$
we get the following lemma:
\begin{lemma}\label{deriv}
We have $\phi'(0)>0$ $(\phi'(0)\approx 0.0615)$. Therefore, the convex function $\phi$ is increasing on $\mathbb{R}^+$.
\end{lemma}

\noindent Eventually, we give the graphs of the functions $\phi$, $\phi'$ and $\Delta$ in Figures 1, 2 and 3.

\begin{center}
\begin{tabular}{c}
\includegraphics[height=7cm,width=6.5cm]{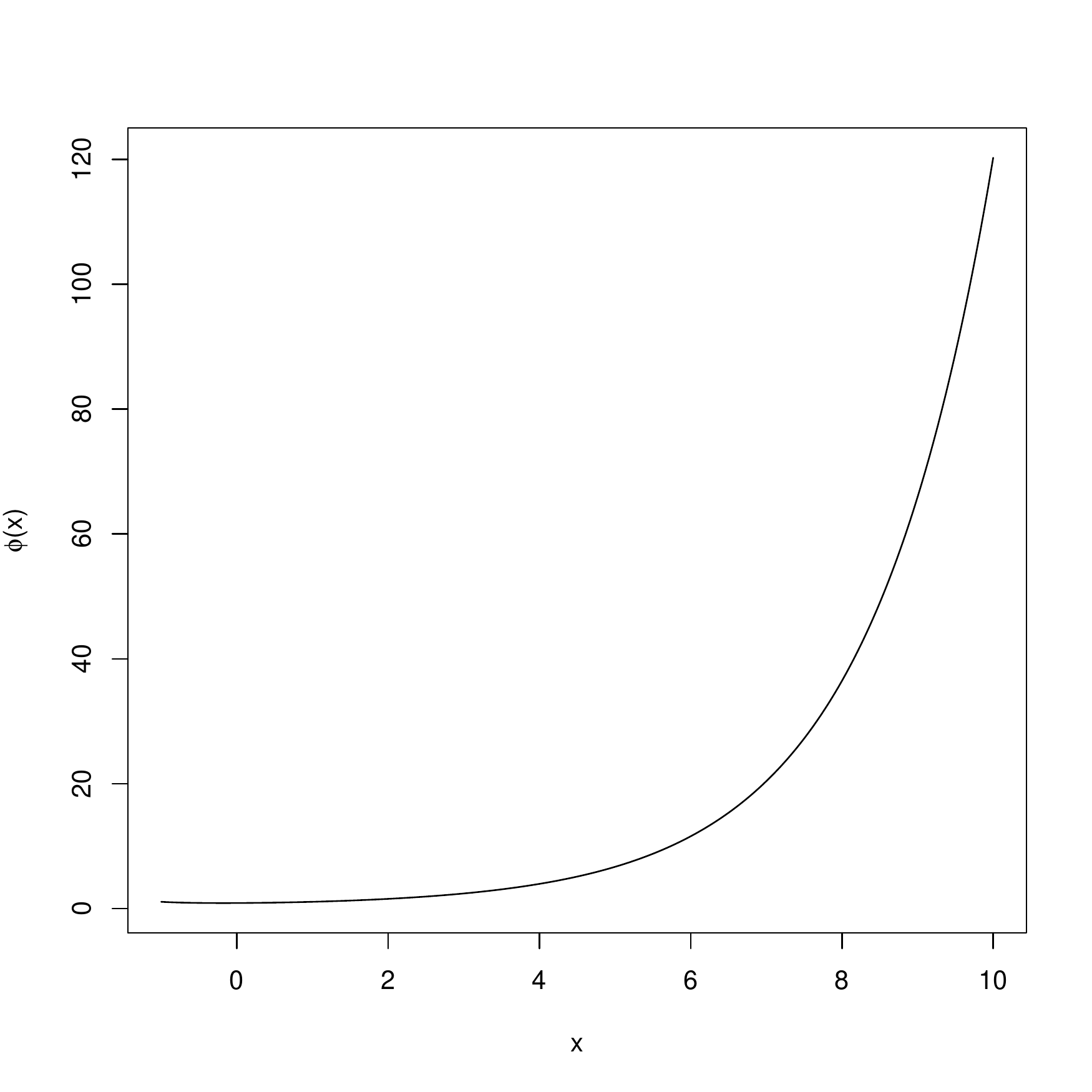}
\end{tabular}
\caption{\textit{Function $\phi$, from $-1$ to $10$.}}
\end{center}
\begin{center}
\begin{tabular}{cc}
\includegraphics[height=7cm,width=6.5cm]{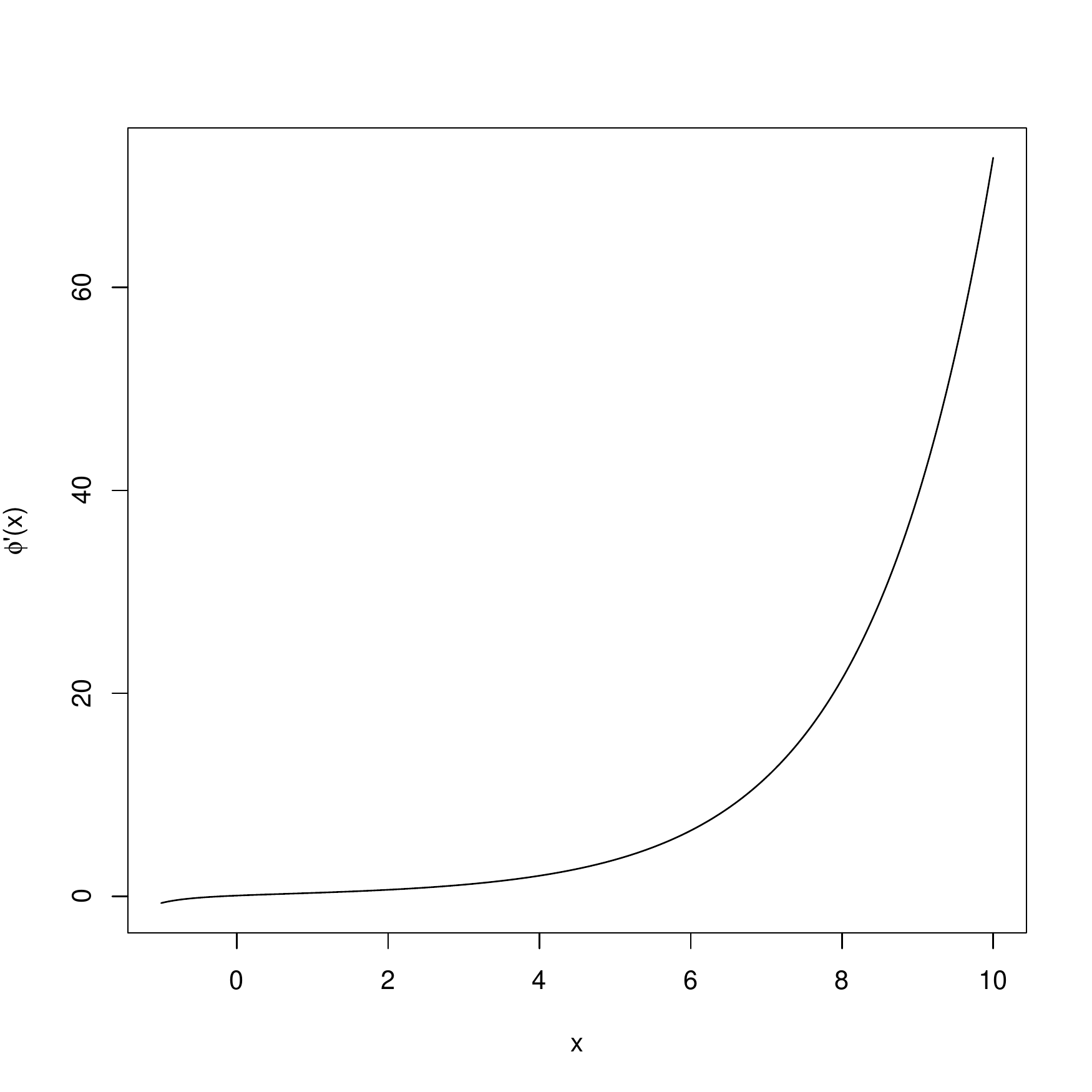}&
\includegraphics[height=7cm,width=6.5cm]{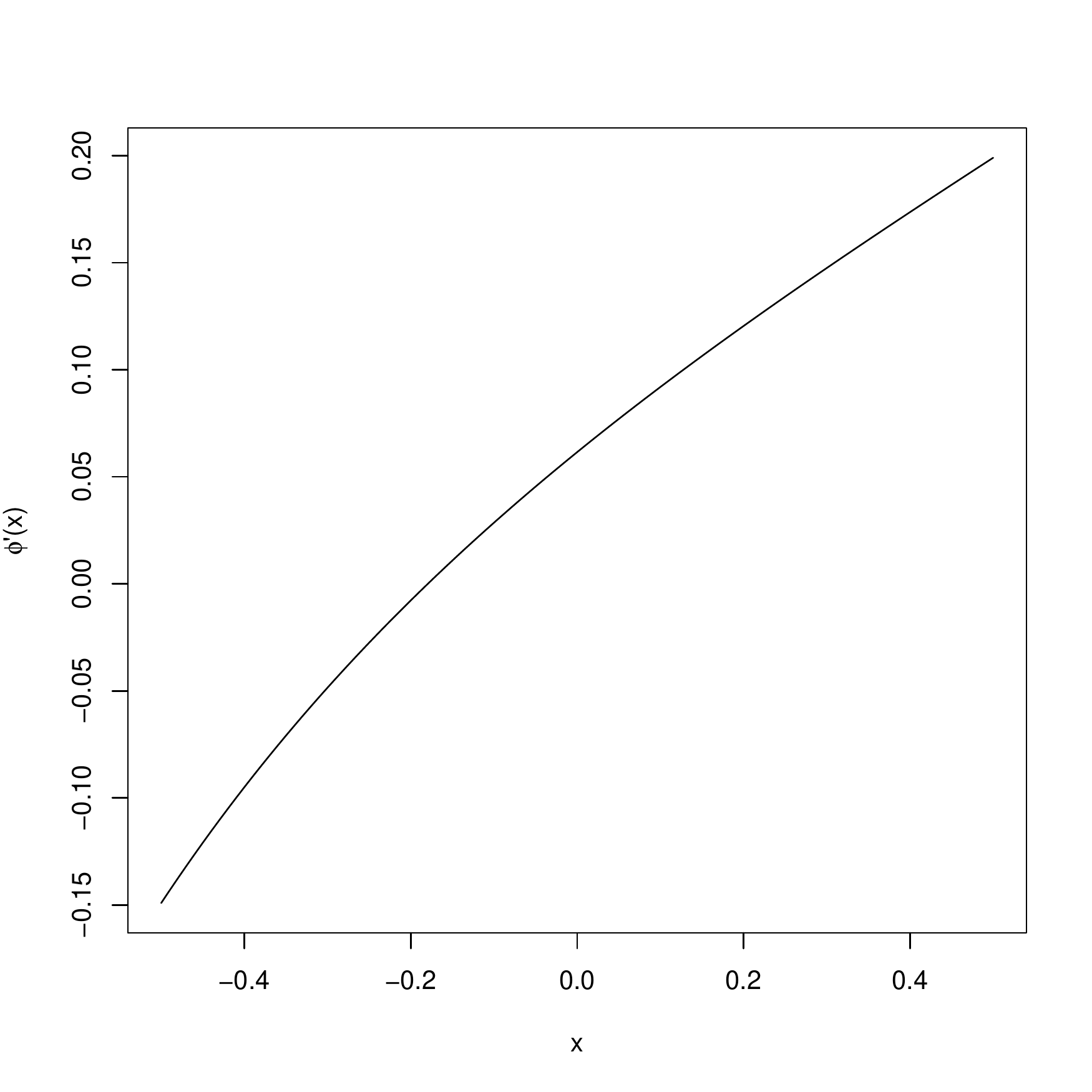}
\end{tabular}
\caption{\textit{Function $\phi'$, from $-1$ to $10$ (left) and
from $-0.5$ to $0.5$ (right).}}
\begin{tabular}{cc}
\includegraphics[height=7cm,width=6.5cm]{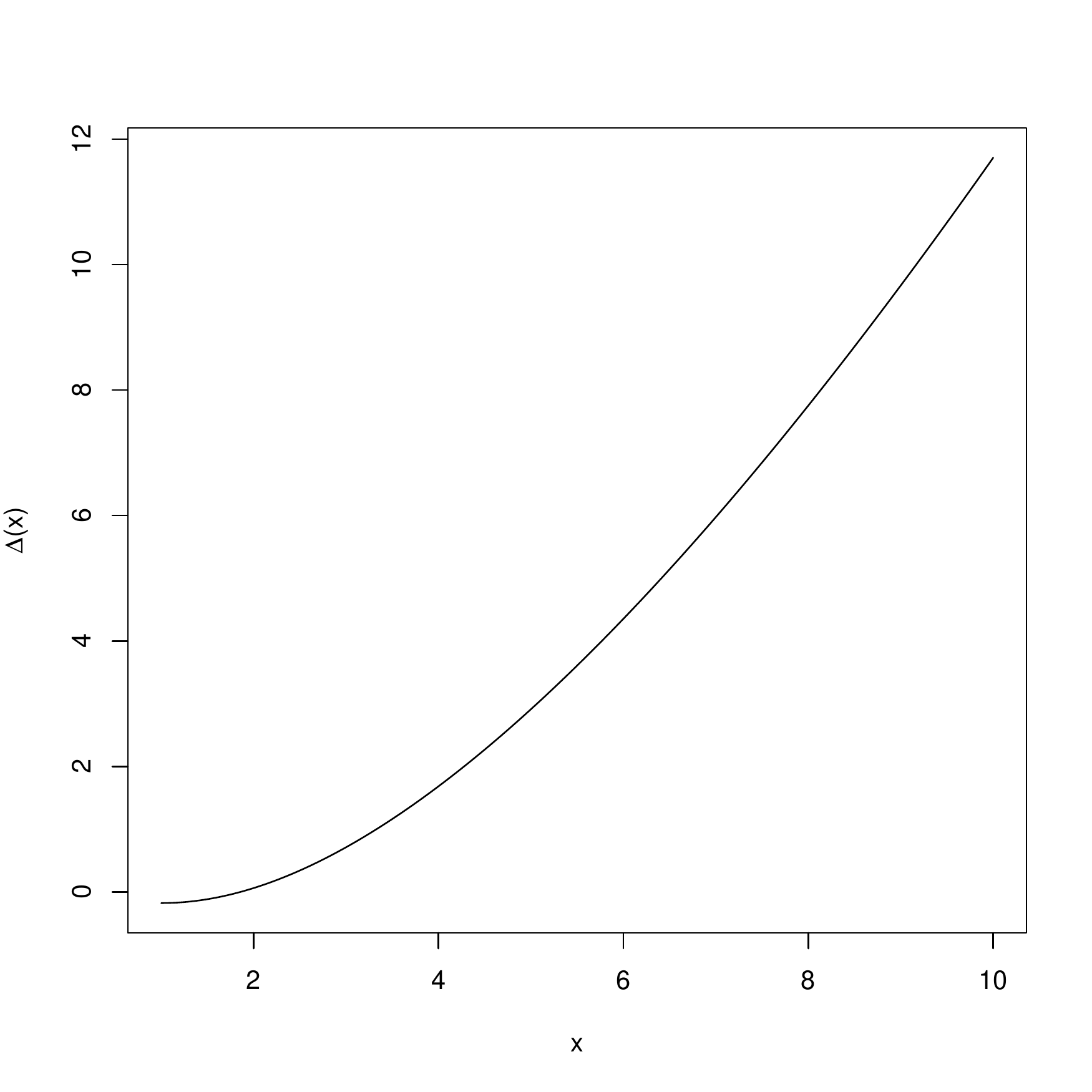}&
\includegraphics[height=7cm,width=6.5cm]{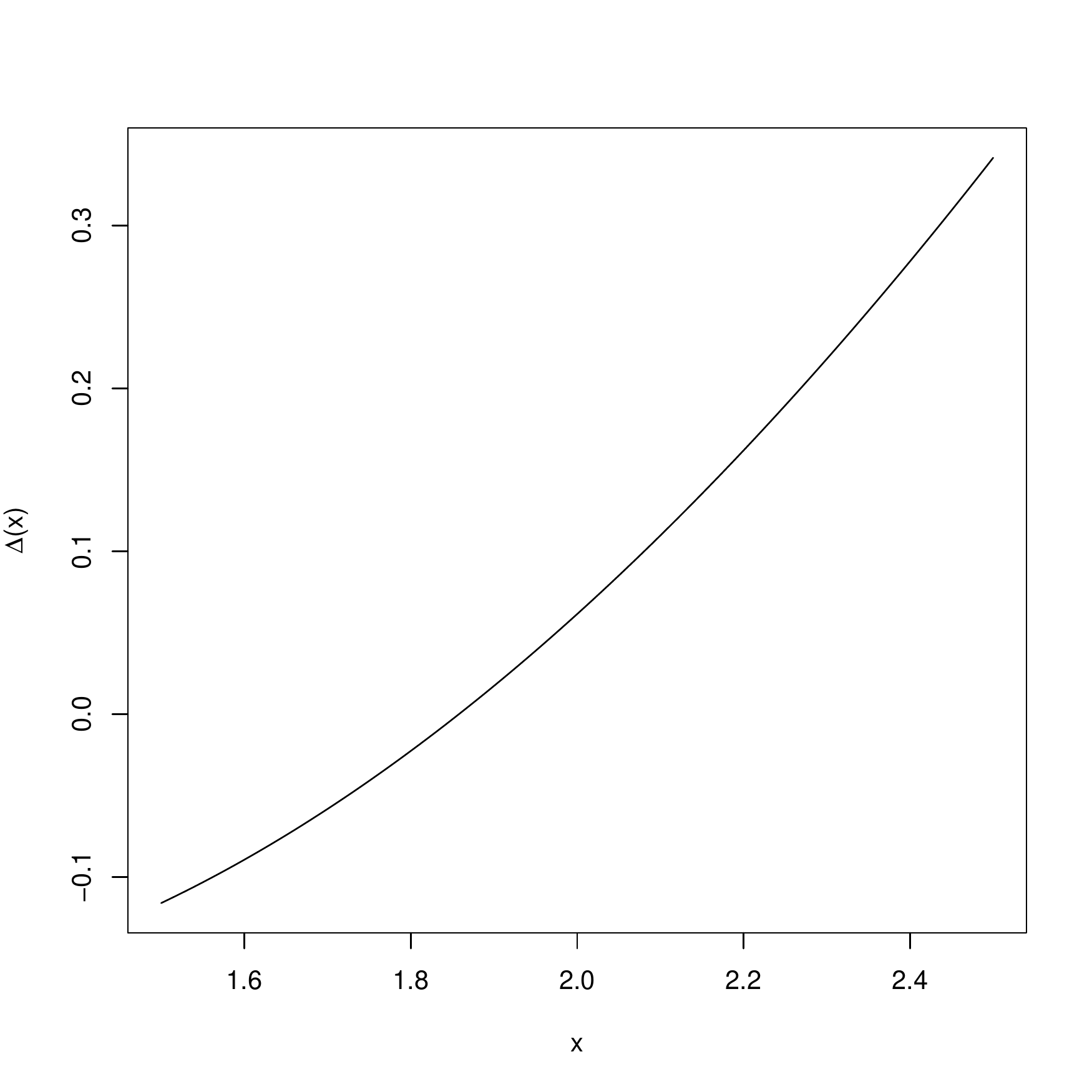}
\end{tabular}
\caption{\textit{Function $\Delta$, from $1$ to $10$ (left) and
from $1.5$ to $2.5$ (right).}}
\end{center}

\section{Another proof of Proposition \ref{transfo}}\label{RK}

In this section, we give a proof of Proposition \ref{transfo} which is based on the Ray-Knight theorem. First note that multiplying both sides of the equalities in Proposition \ref{transfo} by $\e(\mu)$ and using Girsanov's theorem, we see it is equivalent for $0< b< 1$ and $x\geq 0$ to
$$\E[L_{T_1}^{1-b}(\mu)]=\frac{1}{\mu}\big(1-\e(-2\mu b)\big)$$
and
$$\E[L_{T_1}^{-x}(\mu)]=\frac{1}{\mu}\Big(\e(-2\mu x)-\e\big(-2\mu(1+2x)\big)\Big),$$
where $L_{T_1}^{y}(\mu)$ denotes the local time at level $y$ of the Brownian motion with drift $\mu$, $B^{\mu}$, considered up to its first hitting time of $1$. Let us write $X_b=L_{T_1}^{1-b}(\mu)$. Ray-Knight's theorem tells us that for $0< b< 1$, $X_b$ is a (weak) solution of the following stochastic differential equation (SDE):
$$X_b=2\int_0^b\sqrt{X_s}d\beta_s-2\mu\int_0^bX_s ds+2b,$$
where $\beta$ is a Brownian motion, see \cite{borodin2002handbook}, pages 74-79. We now wish to compute $u(b)=\E[X_b]$. From the preceding SDE, we get
$$u(b)=-2\mu\int_0^bu(c)dc+2b.$$
This ordinary differential equation can be easily solved using the variation of the constant method so that we get
$$u(b)=\frac{1}{\mu}\big(1-\e(-2\mu b)\big).$$
The proof for $L_{T_1}^{-x}(\mu)$ goes similarly.

\section*{Acknowledgments}
We thank Vincent Lemaire for helpful comments and computations related to Appendix \ref{fonction}.
\bibliographystyle{abbrv}
\bibliography{bibli_ery}

\begin{thebibliography}{10}

\bibitem{andrews1999special}
G.~E. Andrews, R.~Askey, and R.~Roy.
\newblock {\em Special Functions}.
\newblock Encyclopedia Math. Appl, 1999.

\bibitem{bertoin1994path}
J.~Bertoin and J.~Pitman.
\newblock Path transformations connecting {B}rownian bridge, excursion and
  meander.
\newblock {\em Bulletin des sciences math{\'e}matiques}, 118(2):147--166, 1994.

\bibitem{biane1987processus}
P.~Biane, J.-F. Le~Gall, and M.~Yor.
\newblock Un processus qui ressemble au pont brownien.
\newblock In {\em S{\'e}minaire de Probabilit{\'e}s XXI}, pages 270--275.
  Springer, 1987.

\bibitem{biane1988quelques}
P.~Biane and M.~Yor.
\newblock Quelques pr{\'e}cisions sur le m{\'e}andre brownien.
\newblock {\em Bulletin des sciences math{\'e}matiques}, 112(1):101--109, 1988.

\bibitem{borodin2002handbook}
A.~N. Borodin and P.~Salminen.
\newblock {\em Handbook of {B}rownian motion: facts and formulae}.
\newblock Springer, 2002.

\bibitem{imhof1984density}
J.-P. Imhof.
\newblock Density factorizations for {B}rownian motion, meander and the
  three-dimensional bessel process, and applications.
\newblock {\em Journal of Applied Probability}, pages 500--510, 1984.

\bibitem{knight1988inverse}
F.~B. Knight.
\newblock Inverse local times, positive sojourns, and maxima for {B}rownian
  motion.
\newblock In {\em Colloque Paul L\'evy sur les Processus Stochastiques}, pages
  233--247, 1988.

\bibitem{lachal1991premier}
A.~Lachal.
\newblock Sur le premier instant de passage de l'int{\'e}grale du mouvement
  brownien.
\newblock In {\em Annales de l'institut Henri Poincar{\'e} (B) Probabilit{\'e}s
  et Statistiques}, volume~27, pages 385--405. Gauthier-Villars, 1991.

\bibitem{pitman1999distribution}
J.~Pitman.
\newblock The distribution of local times of a {B}rownian bridge.
\newblock In {\em S{\'e}minaire de probabilit{\'e}s XXXIII}, pages 388--394.
  Springer, 1999.

\bibitem{revuz1999continuous}
D.~Revuz and M.~Yor.
\newblock {\em Continuous martingales and {B}rownian motion}, volume 293.
\newblock Springer, 1999.

\bibitem{yor2001exponential}
M.~Yor.
\newblock {\em On exponential functionals of {B}rownian motion and related
  processes}.
\newblock Springer, 2001.

\end{thebibliography}

\end{document}